\documentclass[11pt]{article}
\usepackage{graphicx}
\usepackage{tabularx}
\usepackage{url}
\usepackage{amsthm} 
\usepackage{amssymb}
\usepackage{amsmath}

\usepackage{color}
\usepackage{xspace}
\usepackage{tikz}
\usetikzlibrary{decorations,decorations.pathmorphing,decorations.pathreplacing,fit,positioning,arrows}

\tikzstyle{w_vertex}=[circle,fill=black!100,text=white,inner sep=0.4mm,draw]
\tikzstyle{vertex}=[circle,fill=black!100,text=white,inner sep=0.8mm]
\tikzstyle{point}=[circle,fill=black,inner sep=0.1mm]

\theoremstyle{plain}
\newtheorem{theorem}{Theorem}
\newtheorem{lemma}{Lemma}

\theoremstyle{definition}
\newtheorem{definition}{Definition}

\theoremstyle{remark}

\date{}

\title{A dichotomy for graphs of bounded degeneracy}

\author{A. Atminas\thanks{Department of Mathematical Sciences, Xi'an Jiaotong-Liverpool University, 111 Ren'ai Road, Suzhou 215123, China. Email: Aistis.Atminas@xjtlu.edu.cn} 
\and V. Lozin\thanks{Mathematics Institute, University of Warwick, Coventry, CV4 7AL, UK. Email:  V.Lozin@warwick.ac.uk}}

\usepackage{caption}
\usepackage{amsmath} 
\usepackage{amsfonts}
\usepackage[cp1251]{inputenc}
\usepackage{colordvi}
\usepackage{graphicx}
\usepackage{tikz}
\usepackage{subcaption}
\usetikzlibrary{decorations.pathreplacing,arrows,automata}
\begin{document}
\maketitle

\newtheorem{obs}{Observation}

\newtheorem{prop}{Proposition}
\newtheorem{cor}{Corollary}

\def\N{\mathbb{N}}
\def\t{\sim}
\def\nt{\nsim}
\def\1{n+1}
\def\2{n+2}

\begin{abstract}
We prove that the degeneracy of graphs in a hereditary class defined by a finite set $S$ of forbidden induced subgraphs is bounded 
if and only if $S$ includes a complete graph, a complete bipartite graph and a forest.  
\end{abstract}

{\it Keywords}: Degeneracy; Hereditary class; Ramsey theory


\section{Introduction}

Degeneracy is a graph parameter between tree-width and chromatic number in the sense that bounded tree-width implies bounded degeneracy, 
while bounded degeneracy implies bounded chromatic number. According to the {G}y\'{a}rf\'{a}s-{S}umner conjecture \cite{Sumner}, 
the chromatic number is bounded in a hereditary class defined by a finite set $S$ of forbidden induced subgraphs (finitely defined class)
if and only if $S$ includes a complete graph and a forest. 
The conjecture was verified for a variety of classes, but in general it is widely open. 

For tree-width, a dichotomy for finitely defined classes was proved in \cite{tree-width}, where it was shown that 
the tree-width of graphs in a hereditary class defined by a finite set $S$ of forbidden induced subgraphs is bounded 
if and only if $S$ includes a complete graph, a complete bipartite graph, a tripod (a forest in which every connected component 
has at most 3 leaves) and the line graph of a tripod. 

In the present paper, we extend the dichotomy for tree-width to a dichotomy for degeneracy. 
We prove that the degeneracy of graphs in a hereditary class defined by a finite set $S$ of forbidden induced subgraphs is bounded 
if and only if $S$ includes a complete graph, a complete bipartite graph and a forest. 
Our result generalizes several partial results on this topic and implies, in particular, that 
the {G}y\'{a}rf\'{a}s-{S}umner conjecture is valid for hereditary classes of graphs excluding an induced complete bipartite subgraph.

In the terminology of Ramsey theory, our result shows that complete graphs, complete bipartite graphs and forests are the only three 
unavoidable structures in large graphs from finitely defined classes. For other connections between degeneracy and Ramsey theory we refer the reader to \cite{Lee}.

The organization of the paper is as follows. In the rest of this section, we introduced basic terminology and notation. Section~\ref{sec:main}
is devoted to the proof of the main result. In Section~\ref{sec:con}, we conclude the paper with an open problem.

All graphs in this paper are finite, undirected, without loops or multiple edges.
A graph $H$ is a {\it subgraph} of a graph $G$ if $H$ can be obtained from $G$ be vertex and edge deletions,
and $H$ is an {\it induced subgraph} of $G$ if $H$ can be obtained from $G$ be vertex deletions only,
and $H$ is a {\it spanning subgraph} of $G$ if $H$ can be obtained from $G$ be edge deletions only.
If $H$ is not an induced subgraph of $G$, then we say that $G$ is $H$-free and that $H$ is a forbidden induced subgraph for $G$.  

A class of graphs is {\it hereditary} if it is closed under taking induced subgraphs. Every hereditary class $\cal X$ can be uniquely 
described in terms of minimal forbidden induced subgraphs, i.e. minimal (with respect to the induced subgraph relation) graphs not in $\cal X$.
If the number of minimal forbidden induced subgraphs for $\cal X$ is finite, we say that $\cal X$ is {\it finitely defined}.
 
Two vertices in a graph $G$ are {\it neighbours} if they are adjacent to each other, and {\it non-neighbours} otherwise. 
The {\it degree} of a vertex is the number of its neighbours.
The {\it degeneracy} of $G$ is the minimum $k$ such that every induced subgraph of $G$ contains a vertex of degree at most $k$.

By $K_n$ and $K_{n,m}$ we denote a complete graph (a clique) with $n$ vertices and a complete bipartite graph (a biclique) with parts of size $n$ and $m$, respectively. 
Also, $C_n$ is a chordless cycle of length $n$. An {\it independent set} in a graph is a subset of pairwise non-adjacent vertices. 

The Ramsey number $R(p,q)$ is the minimum $n$ such that every graph with at least $n$ vertices has either an independent set of size $p$ 
or a clique of size $q$.

Given two disjoint subsets $U$ and $W$ of vertices in a graph $G$, we will say that $U$ and $W$ are {\it complete} to each other if there are all possible edges between them,
and that $U$ and $W$ are {\it anti-complete} to each other if there are no edges between them.

A {\it forest} is a graph without cycles and a {\it tree} is a connected forest. A {\it rooted tree} is a tree with a designated vertex. 
We use standard terminology related to rooted trees, such as {\it child, parent, descendant, ancestor, height of the tree}, etc.
Also, we call the set of vertices of distance $p$ from the root {\it generation} $p$.


\section{Main result}
\label{sec:main}


In this section, we will prove the following theorem.

\begin{theorem}\label{thm:main}
Let $\cal X$ be a hereditary class defined by a finite set $S$ of forbidden induced subgraphs. 
There is a constant bounding the degeneracy of graphs in $\cal X$ if and only if $S$ includes a complete graph, a complete bipartite graph and a forest.
\end{theorem}

The ``only if'' direction is easy and can be proved as follows. First of all, if the class $\cal X$ contains  cliques $K_n$ or bicliques $K_{n,n}$ 
for arbitrarily large values of $n$, then the degeneracy is unbounded in $\cal X$, since the degeneracy of $K_{n}$ is $n-1$ and the degeneracy of  $K_{n,n}$ is $n$. 
Secondly, if no forest is excluded, then each of the finitely many forbidden induced subgraphs contains a cycle, 
and hence $\cal X$ contains the class of $(C_3, C_{4}, \ldots, C_k)$-free graphs for some finite value of $k$. 
According to Erd\H{o}s \cite{Erdos}, the chromatic number is unbounded in the class of $(C_3, C_{4}, \ldots, C_k)$-free graphs for each fixed value of $k$. 
Therefore, the degeneracy is unbounded in this class, as graphs of degeneracy at most $d$ have chromatic
number bounded by $d+1$. Thus, the ``only if'' direction of Thoerem~\ref{thm:main} holds: for a finitely defined class of graphs to be
of bounded degeneracy, we must exclude a clique, a biclique and a forest. The rest of the section will be devoted to proving the ``if'' direction. 
We start with the following lemma.

\begin{lemma}\label{lem:1}
Let $T$ be a tree on $d$ vertices. Then any graph $G$ of degeneracy at least $d$ contains $T$
as a (not necessarily induced) subgraph.
\end{lemma}

\begin{proof}
We prove the lemma by induction on $d$. For $d=1$, the result is obvious.
Now let $G$ be a graph of degeneracy at least $d\ge 2$. Then it contains an induced subgraph $H$
of degeneracy at least $d$ in which every vertex has degree at least $d$. 

Let $T$ be a tree with $d$ vertices. Since $d\ge 2$, $T$ contains a vertex $v$ of degree $1$. 
Let $u$ be the only neighbour of $v$ in $T$ and let $T' = T- v$.
By the inductive assumption, $H$ contains a subgraph isomorphic to $T'$.
Let $x$ be a vertex of $H$, which is the image of $u$ under this isomorphism. Since the degree of $x$ is at least $d$ in $H$
and $T'$ has $d-1$ vertices, vertex $x$ has a neighbour $y$ in $H$ that does not belong to the image of $T'$ in $H$.
Mapping $v$ to $y$ gives us an isomorphism between $T$ and a subgraph of $H$.
\end{proof}

\begin{definition}
Let $T$ be a rooted tree and $k_1,k_2,\ldots,k_n$ be a sequence of natural numbers. We say that $T$ is a $(k_1,k_2,\ldots,k_n)$-tree if
each vertex in generation $0 \leq i<n$ has exactly $k_{i+1}$ children.
If $k_1=k_2=\ldots=k_n=n$, then we say that $T$ is an $n$-tree.  
\end{definition}

To prove the main result of the paper, we will show the following Ramsey-type theorem. 

\begin{theorem}\label{thm:2}
For every $n\in \mathbb N$, there exists an $N\in {\mathbb N}$  such that any graph containing an
$N$-tree as a subgraph, contains either $K_n$ or $K_{n,n}$ or an $n$-tree as an induced subgraph.
\end{theorem}

We note that this result proves the ``if'' direction of Theorem~\ref{thm:main}. Indeed, 
assume the degeneracy of graphs in $\cal X$ is unbounded. Then,  by Lemma~\ref{lem:1}, 
for any $N\in {\mathbb N}$ there is a graph in $\cal X$ that contains an $N$-tree as a subgraph. 
This implies, by Theorem~\ref{thm:2}, that for any $n\in {\mathbb N}$ there is a graph in $\cal X$ 
that contains either a clique $K_n$, a biclique $K_{n,n}$ or an $n$-tree as an induced subgraph.  
Since each forest is an induced subgraph of an $n$-tree for some $n$, we conclude that $\cal X$ contains 
either all cliques or all bicliques or all forests, and therefore, either a clique, a biclique or 
a forest does not appear in the set $S$ of forbidden induced subgraphs for $\cal X$, proving the ``if'' direction.

We will prove Theorem~\ref{thm:2} by Ramsey-type arguments cleaning-up the non-induced copy of the $N$-tree to obtain 
an induced copy of an $n$-tree, or else a clique $K_n$ or an induced biclique $K_{n,n}$ appears. 
The cleaning will be done in two stages producing what we call descendant-only and linked-generation-descendant-only $N$-trees in Sections~\ref{sec:1} and ~\ref{sec:2}, respectively.
Section~\ref{sec:3} will use the cleaned-up version of linked-generation-descendant-only $N$-trees to deduce Theorem~\ref{thm:2}.

\subsection{From large subgraph $N$-trees to large induced descendant-only $M$-trees}
\label{sec:1}

A graph $G$ containing an $N$-tree $T$ as a spanning subgraph will be called an $N$-graph. We will use the two terms, $N$-graph and $N$-tree, interchangeably,
and will extend the terminology used for $T$ (root, parent, child, descendant, generation, etc.) to the graph $G$. The edges of $G$ that do not belong to $T$ will be called {\it non-tree edges}.  

\begin{definition}
We will say that an $N$-graph $G$ is a {\it descendant-only} $N$-tree if every non-tree edge of $G$ connects two vertices one of which is a descendant of the other.  
\end{definition}

\begin{lemma}\label{lem:2}
For any positive integers $k$ and $p$, there is a positive integer $U=U(k,p)$ such that if a graph $G$ contains a family of $U$ disjoint vertex subsets 
$V_1,V_2,\ldots,V_U$ each of size at most $k$, then $G$ contains either an induced $K_p$, an induced $K_{p,p}$ or a subfamily $V_{i_1},V_{i_2},\ldots,V_{i_p}$ of $p$ subsets 
with no edges between any two of them.
\end{lemma}

\begin{proof}
In \cite{tree-width}, it was shown that for all positive integers $a$ and $b$, there is a positive integer $C=C(a,b)$ such that if 
a graph $G$ contains a collection of $C$ pairwise disjoint subsets of vertices, each of size at most $a$ 
and with at least one edge between any two of them, then $G$ contains a $K_{b,b}$-subgraph. Now we apply this result 
to prove the lemma with $U=R(p,C(k,R(p,p)))$. Indeed, if we have a collection of $U$ disjoint subsets, then according to Ramsey, 
either we have $p$ of the subsets with no edges between any two of them, in which case we are done, or we have $C(k,R(p,p))$ subsets with at least one edge between any two of them.
In the latter case, we must have a $K_{t,t}$ subgraph with $t=R(p,p)$. Then among the $t$ vertices in each part of the graph we either find a clique of size $p$, in which case we are done,
or we find an independent set in each part and hence an induced $K_{p,p}$.     
\end{proof}

\begin{theorem}\label{thm:3}
For any positive integer $M$, there is a positive integer $N=N(M)$ such that any $N$-graph contains either
a descendant-only $M$-tree, a clique $K_M$ or a biclique $K_{M,M}$ as an induced subgraph.
\end{theorem}

\begin{proof}
For a sequence of positive integers $p_1,p_2,\ldots,p_t$,  we denote by $S(p_1,p_2,\ldots,p_t)$ the number of vertices in the $(p_1, p_{2}, \ldots, p_t)$-tree, i.e.  
$$S(p_1,p_2,\ldots,p_t)=1+p_1+p_1p_{2}+p_1p_{2}p_{3}+\ldots+p_1p_{2}\cdots p_{t}.$$
Now, given a positive integer $M$, we define the sequence of positive integers $k_1,k_2,\ldots$ recursively as follows:
\begin{itemize}
\item $k_1=R(M,M)$ and 
\item for $t>1$, $k_t=U(S(k_{t-1},k_{t-2},\ldots,k_1),M)$, where $U$ is the function from Lemma~\ref{lem:2}.
\end{itemize}
Next, we show by induction on $t$ that any $(k_t,k_{t-1},\ldots,k_1)$-graph contains either a descen\-dant-only $(M,M,\ldots,M)$-tree of height $t$ rooted
at the same vertex as the original graph, a clique $K_M$ or a biclique $K_{M,M}$ as an induced subgraph.
For $t=1$, the statement is trivial, so assume $t>1$ and let $G$ be a $(k_t,k_{t-1},\ldots,k_1)$-graph with a root vertex $v$. 
By definition, $v$ has $k_t=U(S(k_{t-1}, k_{t-2}, \ldots, k_{1}),M)$ children each of which is the root of a subtree with $S(k_{t-1}, k_{t-2}, \ldots,k_{1})$ vertices. 
Therefore, by Lemma~\ref{lem:2}, $G$ has either an induced $K_M$, an induced $K_{M,M}$ or a collection of $M$ subtrees rooted at the children of $v$ with no edges between any two of them. 
If $G$ contains $K_M$ or $K_{M,M}$, then we are done, so 
consider the case when we have $M$ subtrees rooted at $M$ children of $v$ with no edges between any two of them. 
By the inductive assumption, each of these subtrees contains either an induced $K_M$, an induced $K_{M,M}$ or 
an induced descendant-only $(M,M,\ldots,M)$-tree of height $t-1$ rooted at a child of $v$. 
Once again, if any of these subtrees  contains an induced $K_M$ or an induced $K_{M,M}$, then we are done. 
Otherwise, we have $M$ descendant-only $(M,M,\ldots,M)$-trees of height $t-1$ rooted at $M$ children of $v$, with no edges
from one tree to another, hence $G$ contains an induced descendant-only $(M,M,\ldots,M)$-tree of height $t$.
The result now follows with $N=k_M$.
\end{proof}

\subsection{From descendant-only $M$-trees to linked-generation-descendant-only $K$-trees}
\label{sec:2}

We now introduce an even more restricted version of descendant-only $M$-tree.

\begin{definition}
Let $M$ be a positive integer, and let $T$ be a descendant-only $M$-tree. 
We say that $T$ is a {\it linked-generation-descendant-only} $M$-tree if for every two generations $0 \le i < j \le M$, 
either each vertex of generation $i$ is complete to the set of its descendants in generation~$j$
or each vertex of generation $i$ is anti-complete to the set  of its descendants in generation~$j$. 
\end{definition}

The main result of this section is the following theorem.

\begin{theorem}\label{thm:4} 
For every positive integer $K$, there exists a positive integer $M=M(K)$ such that every descendant-only $M$-tree contains a linked-generation-descendant-only $K$-tree.
\end{theorem}

\begin{proof}
Let $M = K2^{\binom K 2}$. Consider a descendant-only $M$-tree, and let $G$ be the graph containing the first $K$ generations of the tree. 

Let $v$ be the root of $G$ and let $V_i$ be the set of vertices of generation $i$ for $i\ge 2$. 
We colour the neighbours of $v$ in $V_i$ black and non-neighbours of $v$ in $V_i$ white. Then we propagate this colouring to the vertices in 
the previous generations according to the following majority rule: if a vertex in generation $j<i$ has more black children than white, we assign black colour to it, 
otherwise we assign white colour to it (breaking ties arbitrarily). 
Now if the root $v$ is black, then we remove from the graph all white vertices and all
their descendants, and if $v$ is white we remove all black vertices and all their descendants. 
After this procedure, which we call cleaning the relationship between $v$ and generation $i$, vertex $v$ is either complete or anti-complete to 
all vertices of generation $i$. It is not difficult to see that every vertex that survived in this procedure has lost at most half of its children.

More generally, for cleaning the relationship between $V_i$ (generation $i$) and $V_j$ (generation $j$) with $i<j$ and $j-i\ge 2$,
we colour a vertex $u$ in $V_j$ black if its ancestor in $V_i$ is adjacent to $u$ and we colour it white otherwise. Then we propagate 
this colouring for all generations smaller than $j$ (including the root) using the majority rule, and then delete the vertices (and all their
descendants) of the colour which is not assigned to the root. Again, each of the vertices that survived this procedure has lost at most half of
its children. Therefore, after applying the cleaning procedure to each pair of non-consecutive generations, we a left with a tree in which each
non-leaf vertex has at least $M/2^{\binom K 2} \geq K$ children, providing us with a desired linked-generation-descendant-only $K$-tree. 
\end{proof}

\subsection{Proof of Theorem~\ref{thm:2}}
\label{sec:3}

A linked-generation-descendant-only $K$-tree produced in the previous section can be u\-nique\-ly described by an auxiliary graph,
which we call {\it generation graph}, with vertices $\{0,1,\ldots, K\}$ and edges
connecting two generations  $i$ and $j$ if and only if generation $i$ vertices are complete to their descendants 
in generation $j$ (generation $i$ is ``linked'' to generation $j$).
Observe that this graph contains a (not necessarily induced) path of length $K$ joining $i$ to $i+1$ for all $i = 0,1, 2,\ldots, K-1$.

If $K$ is sufficiently large, then, assuming that the $K$-tree is free of large cliques and bicliques, the generation graph contains 
a long induced (chordless) path corresponding to an increasing sequence of generations. This can be proved directly by showing that 
every generation $i$ in the $K$-tree is linked to a bounded number of generations $j$ with $j<i$, since otherwise a large biclique arises. 
We, instead, quote the following result from \cite{Ramsey}. In this result, a graph is called {\it traceable} if it contains a Hamiltonian path, i.e. 
a path $P$ containing all vertices of the graph, and a sub-path of $P$ is called {\it increasing} if its vertices respect the order in which they appear in $P$.  

\begin{theorem}{\rm \cite{Ramsey}}\label{thm:Ramsey}
For any positive integer $n$, there is a positive
integer $K=K(n)$ such that every traceable graph with  $K$
vertices contains an induced $K_{n}$, an induced $K_{n,n}$ or an induced increasing path with $n$ vertices.
\end{theorem}

\medskip
We are now ready to prove Theorem~\ref{thm:2}.

\begin{proof}
Let $n$ be an arbitrary positive integer. Let $K = K(n)$ be the number from Theorem~\ref{thm:Ramsey}, 
let $M = M(K)$ be the number from Theorem~\ref{thm:4} and let $N=N(M)$ be the number from Theorem~\ref{thm:3}.
Consider an $N$-tree. By Theorem~\ref{thm:3} we know that either this $N$-tree contains an induced biclique 
$K_{M,M}$ or a  clique $K_M$, in which case we are done, or we have a descendant-only $M$-tree. 
Now, by Theorem~\ref{thm:4}, the descendant-only $M$-tree
contains a linked-generation-descendant-only $K$-tree $T$, and it can be uniquely represented
by a $K$-vertex generation graph $G$, which tells us about the relationships between different
generations. This graph, and hence the $K$-tree $T$, has a path of length $K$ (such a path in $T$ can be composed of the edges of the tree connecting  
a vertex in generation $K$ to the root). 
Therefore, by Theorem~\ref{thm:Ramsey}, 
both $G$ and $T$ contain either a clique $K_n$ or an induced biclique $K_{n,n}$ or an induced increasing path with $n$ vertices. 
In case of a clique or a biclique in $T$, we are done, while  an induced increasing path with $n$ vertices
gives rise to an induced $n$-tree. This finishes the proof.
\end{proof}


\section{Conclusion}
\label{sec:con}

To bound the degeneracy in a hereditary class, we obviously must exclude (forbid) a clique and a biclique.
Also, we must exclude a graph from  the class of $(C_3, C_{4}, \ldots, C_k)$-free graphs for each value of 
$k$, since the degeneracy is unbounded in these classes. The only way to do this by means of finitely many forbidden induced subgraphs 
is forbidding a graph from the intersection of all these classes,
i.e. forbidding a forest. In the present paper we have shown that forbidding a forest together with a clique and a biclique is sufficient for bounding degeneracy.
If we do not forbid a forest, then for bounding degeneracy we need to exclude infinitely many graphs. Only partial results are available for such classes.
In particular, degeneracy is bounded for $(C_k, C_{k+1}, \ldots)$-free graphs excluding a clique and a biclique, which follows from the results in \cite{KO}
(a specific bound on degeneracy was also obtained in \cite{deg}).
However, forbidding large cycles is not the only way to destroy classes of $(C_3, C_{4}, \ldots, C_k)$-free graphs for all values of $k$.
Identifying all sufficient conditions for bounding degeneracy in a hereditary class remains a challenging open problem.


\begin{thebibliography}{99}

\bibitem{deg}
M. Bonamy, N. Bousquet, M. Pilipczuk, P. Rza\.{z}ewski,  S. Thomass\'{e},  
B. Walczak,  Degeneracy of {$P_t$}-free and {$C_{\geqslant t}$}-free graphs
              with no large complete bipartite subgraphs. {\it J. Combin. Theory Ser. B} 152 (2022) 353--378.


\bibitem{Sumner}
M. Chudnovsky, P. Seymour, 
Extending the {G}y\'{a}rf\'{a}s-{S}umner conjecture.
{\it J. Combin. Theory Ser. B} 105 (2014) 11--16.

\bibitem{Ramsey}
F. Galvin, I. Rival, B. Sands, A Ramsey-type theorem for traceable graphs. {\it J. Combin. Theory Ser. B} 33 (1982) 7--16.

\bibitem{Erdos}
P. Erd\H{o}s,  Graph theory and probability. {\it Canadian J. Math.} 11 (1959) 34--38.

\bibitem{KO}
D. K\"{u}hn, D. Osthus,  
Induced subdivisions in $K_{s,s}$-free graphs of large average degree. 
{\it Combinatorica} 24 (2004) 287--304.

\bibitem{Lee}
C. Lee,  Ramsey numbers of degenerate graphs. {\it Annals of Mathematics}. 185 (2017) 791--829.

\bibitem{tree-width}
V. Lozin, I. Razgon, Tree-width dichotomy. {\it European J. Combinatorics}. 103 (2022) 103517. 

\end{thebibliography}
\end{document}